\documentclass[11pt,letterpaper]{amsart}
\usepackage{mathtools,amssymb,amsthm,mathrsfs,enumitem}
\usepackage[T1]{fontenc}
\usepackage[utf8]{inputenc}
\usepackage{geometry}
\usepackage{xr-hyper}
\usepackage[colorlinks,linkcolor=blue,citecolor=blue,urlcolor=blue]{hyperref}

\newtheorem{theorem}{Theorem}[section]
\newtheorem{lemma}[theorem]{Lemma}
\newtheorem{proposition}[theorem]{Proposition}
\theoremstyle{definition}
\newtheorem{remark}[theorem]{Remark}
\numberwithin{equation}{section}
\newcommand{\I}{\mathbf I}
\newcommand{\T}{\mathbb T}
\newcommand{\R}{\mathbb R}
\newcommand{\C}{\mathbb C}
\newcommand{\Z}{\mathbb Z}
\newcommand{\Q}{\mathbb Q}
\newcommand{\rmm}[1]{\mathrm{#1}}

\newcommand{\tnorm}[1]{\left\vert\mkern-1.5mu\left\vert\mkern-1.5mu\left\vert #1
\right\vert\mkern-1.5mu\right\vert\mkern-1.5mu\right\vert}

\title[DTMP in the subcritical Type I regime]{Dry Ten Martini Problem in the Subcritical Type I Regime}
\date{}

\author{Xianzhe Li}
\address{Department of Mathematics, University of California, Berkeley, CA 94720, USA} 
 \email{xianzhe@berkeley.edu}   

\begin{document}
\maketitle
\begin{abstract}

We establish a resolvent factorization of the hyperbolic projection
associated with the finite-range dual
operators, by a kernel with exponential off-diagonal decay.
As an application, we prove that, for every irrational frequency and every analytic potential, each subcritical Type I energy satisfying the gap-labelling condition is an
endpoint of an open spectral gap. This confirms the conjecture of Ge--Jitomirskaya--You \cite{GJY, You} in the subcritical regime. 

\end{abstract}
\section{Introduction}
Consider the one-frequency quasiperiodic Schr\"odinger operator
\[
(H_{v,\alpha,x}u)_n=u_{n+1}+u_{n-1}+v(x+n\alpha)u_n,
\qquad n\in\Z,
\]
where $\alpha\in\R\setminus\Q$, $x\in\T$, and
$v\in C^\omega(\T,\R)$. Its spectrum $\Sigma_{v,\alpha}$ is independent
of $x$. The integrated density of states $N_{v,\alpha}$ is constant on
each connected component of $\R\setminus\Sigma_{v,\alpha}$, and the Gap
Labelling Theorem \cite{gaplabel} assigns to every bounded spectral gap
an integer $\mathbf k\ne0$ through
\[
N_{v,\alpha}(E)\equiv \mathbf k\alpha\pmod{\Z}.
\]
The original Dry Ten Martini Problem (DTMP) concerns the almost Mathieu
operator, for which $v(x)=2\lambda\cos 2\pi x$, and asks whether, for
every irrational $\alpha$ and every $\lambda\ne0$, all gaps allowed by
this relation are open. Equivalently, one must rule out collapsed
labelled gaps. For a general analytic potential, the analogous question
is naturally posed for a fixed pair $(v,\alpha)$.
The DTMP and its variants have been extensively studied in the literature; see, for instance, \cite{CEY,P,AJ05,AJ08,LY2015,AYZ,GJY,LXZ,ABD,DGY,BBL,GeWXu}. We refer to \cite[Introduction]{LXZ} for a systematic historical overview and further references.
 Of particular relevance
to the present setting, Argentieri and Avila \cite{AA} showed that
arbitrarily small analytic perturbations in the subcritical regime may
collapse gaps and even produce interval spectrum. Thus one cannot expect
an unrestricted analytic stability statement in this regime, and the
Type I assumption is essential here.

The dynamical object associated with $H_{v,\alpha,x}$ is the
Schr\"odinger cocycle $(\alpha,S_E^v)$, where
\[
S_E^v(x)=\begin{pmatrix}E-v(x)&-1\\1&0\end{pmatrix}.
\]
Write $L_\varepsilon(E)$ for the Lyapunov exponent of its complexification
$S_E^v(\cdot+i\varepsilon)$. An energy in the spectrum is called subcritical if $L_\varepsilon(E)=0$ for all sufficiently small $|\varepsilon|$, supercritical if $L(E)>0$, and critical otherwise. Following
Ge--Jitomirskaya--You \cite{GJY}, one defines the first turning point
$\varepsilon_1(E)$ of $L_\varepsilon(E)$ and the $T$-acceleration
$\bar\omega(\alpha,S_E^v)$ as the quantized slope immediately to its right. An energy is called \emph{Type I} if its \(T\)-acceleration equals one, and the set of all Type I spectral energies is denoted by \(\Sigma_{v,\alpha}^1\). For a  fixed $\alpha$, a potential, and hence the corresponding operator, is called Type I if every energy in its spectrum is of Type I.

Type I operators were introduced as a stable class for which the dual
long-range dynamics has a two-dimensional center. Ge--Jitomirskaya--You
proved the corresponding Ten Martini statement \cite{GJY2} and formulated the Type I
conjecture that every gap allowed by the Gap Labelling Theorem is open
\cite{GJY,You}. In the subcritical regime, \cite{GJY} proved this for
trigonometric polynomial potentials under the arithmetic assumption
$\beta(\alpha)=0$. They have recently proved the full subcritical statement; see \cite{GJY2}. The supercritical part of the conjecture, for
every irrational frequency and general analytic potentials, was proved in
\cite{LXZ}.

We restate the subcritical result below and record a different proof. 
The main differences from the proof will be explained later.

\begin{theorem}\label{thm:subcritical}
Let $\alpha\in\R\setminus\Q$ and let $v$ be an analytic 
potential. Then every energy $E_{\mathbf k}\in\Sigma_{v,\alpha}^{1}$
such that $(\alpha,S_{E_{\mathbf k}}^v)$ is subcritical and
\[
N_{v,\alpha}(E_{\mathbf k})\equiv\mathbf k\alpha\pmod{\Z}
\]
is an endpoint of an open spectral gap.
\end{theorem}
Thus the Type I conjecture holds at every noncritical spectral energy
when the present theorem is combined with the supercritical result of
\cite{LXZ}. The critical regime is still open.

\subsection{Brief review and new ingredient}
The dynamical approach to the DTMP goes back to Puig
\cite{P,P06}. For the almost Mathieu operator with Diophantine frequency and small coupling constant,
he combined Aubry duality, reducibility, and the Moser--P\"oschel argument
to prove the opening of the labelled gaps. Avila--You--Zhou later developed
this framework into a quantitative scheme based on quantitative almost
reducibility and quantitative Aubry duality, extending the argument to all
irrational frequencies \cite{AYZ}, and the whole non-critical regime. 

The special feature of the Type I regime is
that the cocycle associated with the long-range dual operator is partially
hyperbolic with a two-dimensional center. The proof therefore rests on the
intrinsic-center representation of Ge--Jitomirskaya \cite{GJ}, a major
advance in the spectral theory of analytic one-frequency quasi-periodic
long-range operators. Although the dual equation is infinite-range and has
no finite-dimensional transfer matrix, they constructed an intrinsic
finite-dimensional Hermitian symplectic dynamics by identifying the center bundles of
finite-range approximants and proving their convergence. Their
two-dimensional center and its resolvent representation are essential
inputs here. See also \cite{LWYZ} for a broader operator-theoretic treatment
of canonical center dynamics in the multi-frequency setting.

For general analytic potentials, the supercritical Type I argument of
\cite{LXZ} introduced dimension-free quantitative Aubry duality and a cone
argument based on the monotonicity of the reconstructed center cocycle,
replacing the more explicit Moser--P\"oschel computation. At a labelled
subcritical gap boundary, quantitative almost reducibility produces an
almost-parabolic normal form \cite{Avi2023KAM,AYZ}. The remaining issue is to obtain, uniformly in the truncation, a lower bound on the parabolic coefficient relative to the error. In the classical setting, this is precisely where Puig's
dual argument enters.

In the subcritical Type I setting, however, the failure of simplicity of the dual eigenvalue prevents one from obtaining such a lower bound directly.
The argument of Ge--Jitomirskaya--You keeps two approximate solutions in the
full $2d$-dimensional state space and uses delicate quantitative relations
with the stable and unstable directions; see \cite[p.~100]{GJY}. This yields a
nondegeneracy statement for the center components of the approximate
solutions and ultimately allows one to exploit how approximate solutions evolve
under the two-dimensional center dynamics.

Our argument instead uses the factorization \eqref{kernal} below,
without explicitly identifying or estimating the hyperbolic directions.
As observed in \cite{LXZ}, $G_d$ constructed in
\cite{GJ} encodes the center intrinsically and
\(
(G_dS_d)^2=G_dS_d,
\)
so that $G_dS_d$ is the projection onto the center.
Define the inclusion of the window centered at $k$ by
\(
\Lambda_k:\C^{2d}\longrightarrow\ell^2(\Z)
\),
\(
\Lambda_ke_{d-j}=\delta_{k+j}
\),
\(-d\leq j\leq d-1\).
To study a forced dual equation
\(
(L_{v_d,\alpha,\theta}-E)u=r,
\)
we restrict to a fixed window and set
\(
\mathcal U_k=\Lambda_k^*u.
\)
Its center component is
\(
G_d(E,\theta+k\alpha)S_d\,\mathcal U_k,
\)
so the error in recovering $u(k)$ from its center coordinates is the
corresponding coordinate of the hyperbolic component:
\(
\zeta(k)
=
e_d^*\bigl(I_{2d}-G_d(E,\theta+k\alpha)S_d\bigr)\mathcal U_k.
\)
To estimate this error in terms of the forcing $r$, we seek (in
Proposition~\ref{prop:backward-cut-kernel}) a row kernel with exponential off-diagonal decay 
satisfying
\begin{equation}\label{kernal}
    \mathcal K^{E,d}(\theta;k,\cdot)(L_{v_d,\alpha,\theta}-E)
=
e_d^*\bigl(I_{2d}-G_d(E,\theta+k\alpha)S_d\bigr)\Lambda_k^*.
\end{equation}
Applying this identity to $u$ then gives
\(
\zeta(k)
=
\sum_{m\in\mathbb Z}
\mathcal K^{E,d}(\theta;k,m)r(m).
\)
Thus, whenever $r$ is exponentially small, the hyperbolic component
$\zeta(k)$ is exponentially small as well. Hence an approximate dual
solution is recovered from its center component up to an exponentially
small error.
Once a dominated two-dimensional center dynamics has been isolated, the
remaining argument is essentially the usual all frequency Puig argument \cite{AYZ}.

The resolvent factorization \eqref{kernal} actually applies to any gap between
consecutive turning points, irrespective of its slope, i.e., the center dimension.
It may therefore be of independent interest, and we believe it could be
useful in more general contexts, which we will try to explore in future.

\subsection{Organization of the paper}
Section~2 introduces the notation and the objects used in the proof.
Section~3 presents the energy-dependent center construction with
estimates uniform in the truncation degree. Section~4 proves the
quantitative Aubry-duality estimate using the hyperbolic projection
described above. Section~5 combines this estimate with quantitative
almost reducibility to complete the proof of
Theorem~\ref{thm:subcritical}. The appendix derives the resolvent factorization \eqref{kernal}
from the shifted dual resolvents.

\subsection*{AI Statement}
The author used LLMs to assist with manuscript editing. All the main ideas in this paper came from the human author. 

\subsection*{Acknowledgements} The author was partially supported by NSF DMS-2052899, DMS-2155211, Simons 
896624 and an AMS-Simons Travel Grant.

\section{Preliminaries}
For $h>0$, set
\(
\T_h:=\{\theta\in\C/\Z:|\Im\theta|<h\}.
\)
We use $^*$ to denote conjugate transpose; vector norms are
Euclidean and matrix norms are the induced operator norms.
Let $e_1,\ldots,e_{2d}$ and $(\delta_j)_{j\in\Z}$ be the standard
orthonormal bases of $\C^{2d}$ and $\ell^2(\Z)$, respectively.
For a finite-dimensional normed space $X$, let
$C_h^\omega(\T,X)$ denote the $X$-valued bounded holomorphic functions  on
$\T_h$, equipped with
\(
\|F\|_h:=\sup_{\theta\in\T_h}\|F(\theta)\|.
\)
For $h=0$, this means $\|F\|_0:=\sup_{\theta\in\T}\|F(\theta)\|$.
For a family $F^E$ continuous in $E\in\I$, write
\(
\|F\|_{\I,h}:=\sup_{E\in\I}\|F^E\|_h.
\)
For real- or group-valued analytic functions, the target condition is
imposed on the real torus and the function is extended holomorphically
to $\T_h$. Set
\(
 \mathcal J_2=\begin{pmatrix}0&-1\\1&0\end{pmatrix}\), and \(S_d=(S_d)_{j,\ell}
 :=(\mathbf1_{j<0}-\mathbf1_{\ell<0})
 \widehat v_{\ell-j}\mathbf1_{|\ell-j|\leq d},\ 
-d\leq j,\ell\leq d-1.
\)We define the corresponding Hermitian
symplectic forms by
\(
\psi_d(u,w):=u^*S_dw\),
\(\omega_1(a,b):=a^*\mathcal J_2b.
\) We therefore consider the collections
\begin{align*}
\rmm{HSp}(2)
&:=\{M\in\rmm{GL}(2,\C):M^*\mathcal J_2M=\mathcal J_2\},\\
\rmm{HSp}(2d,\psi_d)
&:=\{A\in\rmm{GL}(2d,\C):A^*S_dA=S_d\},\\
\rmm{HSp}(2d\times2,\psi_d,\omega_1)
&:=\{F\in\C^{2d\times2}:F^*S_dF=\mathcal J_2\}.
\end{align*}
For $v\in C^\omega(\T,\R)$, denote its $d$-th truncation by
\(
v_d(x)=\sum_{|j|\leq d}\widehat v_j e^{2\pi i jx}.
\)
Then the corresponding long-range dual operator is given by
\[
(L_{v_d,\alpha,x}u)_n
=\sum_{j=-d}^{d}\widehat v_j u_{n+j}
+2\cos 2\pi(x+n\alpha)u_n.
\]
Assume $\widehat v_d\ne0$. For a solution of
$L_{v_d,\alpha,x}u=Eu$, the long-range transfer matrix $L^{E,d}$ is defined by
\[
U_{n+1}=L^{E,d}(x+n\alpha)U_n,\qquad U_n=(u_{n+d-1},\ldots,u_n,u_{n-1},\ldots,u_{n-d})^\top.
\]
So its state coordinates are ordered as
$d-1,\ldots,0,-1,\ldots,-d$. 
Then
$L^{E,d}\in \rmm{HSp}(2d,\psi_d)$.

For a matrix-valued function $F$ with rows $F(\theta,k)$ indexed by
$-d\leq k\leq d-1$, define the weighted analytic norms
\[
\tnorm{F}_{h}^\eta
:=\sum_{k=-d}^{d-1}e^{-2\pi\eta|k|}
  \|F(\cdot,k)\|_h,
\qquad
\tnorm{F}_{\I,h}^\eta
:=\sup_{E\in\I}\tnorm{F^E}_{h}^\eta.
\]
When $d'>d$, the norm
$\tnorm{F^d-F^{d'}}_{\I,h}^\eta$ is taken over the common rows
$-d\leq k\leq d-1$. This is the weighted norm of \cite{LXZ}, with
weight parameter $\xi=2\pi\eta$.

\section{Center Convergence}

Throughout, all truncation degrees below satisfy $\widehat v_d\ne0$.
If $v$ is a trigonometric polynomial, keep its actual degree fixed
and use the corresponding fixed-dimensional center. The convergence
and endpoint approximation steps are then unnecessary. We write the
proof below for the infinite-range case; its fixed-degree version
uses the same estimates. 
We firstly record the quantitative, energy-dependent part of the center
construction originally used in \cite[Theorem~6.3]{GJ}.

\begin{theorem}[\cite{GJ}]\label{theorem-main-general}
Let $\alpha\in\R\backslash\Q$, $v\in C_{h_1}^\omega(\T,\R)$, and $E_*\in \R$ with
$\bar{\omega}(\alpha,S_{E_*}^v)=1$ and $\varepsilon_1(E_*)>0$.
Then there exist a sufficiently small compact interval $\I$ containing $E_*$
in its interior, $h_0>0$ and $d_0$ such that, for every $d\geq d_0$
with $\widehat v_d\ne0$, there are families, continuous in $E\in\I$,
\[
\mathcal{O}^{E,d}\in C_{h_0}^\omega
(\T,\rmm{HSp}(2d\times 2,\psi_d,\omega_1)),\qquad
\mathcal M^{E,d}\in C_{h_0}^\omega(\T,\rmm{HSp}(2)),
\]
such that
\begin{equation}\label{eq:center-invariance}
L^{E,d}(x)\mathcal{O}^{E,d}(x)
=\mathcal{O}^{E,d}(x+\alpha)\mathcal M^{E,d}(x),
\end{equation}
where $\mathcal M^{E,d}\rightarrow\mathcal M^E$ uniformly for
$E\in\I$ in the $\|\cdot\|_{h_0}$ norm, and $E\mapsto\mathcal M^E$
is smooth.
Furthermore, if we write
\[
\mathcal{O}^{E,d}(x)=\begin{pmatrix}
\mathcal{O}^{E,d}(x,d-1) &
\mathcal{O}^{E,d}(x,d-2) &
\cdots &
\mathcal{O}^{E,d}(x,-d)
\end{pmatrix}^\top,
\]
then there exist $0<\gamma<\varepsilon<\eta<h_1-\gamma$ and $K\geq1$, depending only on
$\alpha,v,h_1,\I$ and independent of $d$, such that the following properties hold:
\begin{enumerate}[label=(H\arabic*), leftmargin=*, itemsep=3pt]
\item
\(
\displaystyle
\tnorm{\mathcal O^{E,d}}_{\I,h_0}^\eta\leq K,
\qquad
\tnorm{\mathcal O^{E,d}-\mathcal O^{E,d'}}_{\I,h_0}^\eta
\leq K\|v_d-v_{d'}\|_{h_1-\gamma},\quad d'>d\geq d_0.
\)
\item $\|\mathcal M^{E,d}-\mathcal M^E\|_{\I,h_0}
\leq K\|v_d-v\|_{h_1-\gamma}$.
\item For every $E\in\I$, the shifted dual resolvents
\(
R_\pm^{E,d}(\theta)
:=\bigl(L_{v_d(\cdot\pm i\varepsilon),\alpha,\theta}-E\bigr)^{-1}
\)
are holomorphic in $\theta$
on $\T_{h_0}$ and satisfy
\(
\|R_\pm^{E,d}(\theta)\|_{\I,h_0}\leq K.
\)
\item For $(E,\theta)\in\I\times\T$ and
$\ell,m\in\Z$, define
\begin{align*}
    g_{v_d(\cdot\pm i\varepsilon)}^\ell(E,\theta,m)
&:=\langle\delta_m,R_\pm^{E,d}(\theta)\delta_\ell\rangle \\
u_d^\ell(E,\theta,m)
&:=e^{-2\pi\varepsilon(m-\ell)}
 g_{v_d(\cdot+i\varepsilon)}^\ell(E,\theta,m)-e^{2\pi\varepsilon(m-\ell)}
 g_{v_d(\cdot-i\varepsilon)}^\ell(E,\theta,m).
\end{align*}
Then for $-d\leq m,\ell\leq d-1$, 
\[
\begin{split}
(G_d)_{m,\ell}(E,\theta)
&:=-u_d^\ell(E,\theta,m)=-\Bigl[\mathcal O^{E,d}(\theta)\mathcal J_2
  \mathcal O^{E,d}(\theta)^*\Bigr]_{m,\ell}.
\end{split}
\]
\end{enumerate}
\end{theorem}

\begin{proof}
All details of the construction are contained in the adapted version
\cite[Theorem~9.2]{LXZ}; we only explain the choice of parameters used
in the present statement. First choose a fixed shift $\varepsilon$ with
\(
\varepsilon_1(E_*)<\varepsilon<\min\{\varepsilon_2(E_*),
h_1\},
\)
where $E_*$ is the energy in the hypothesis and $\varepsilon_2(E_*)$
is the next turning point, with $h_1$ used as the upper endpoint if
there is no second turning point in the analytic strip. After shrinking
$\I$ around $E_*$ and increasing $d_0$, the same strict separation holds
uniformly for all $E\in\I$, both for $v$ and for every $v_d$ with
$d\geq d_0$.
Let $\rho>0$ be a common lower bound for the distance from $\varepsilon$
to the two endpoints of these affine segments. Choose $h_0>0$ so that
$4h_0<\rho$.  After shrinking the complex
energy neighborhood, the shifted resolvents are therefore uniform on
$\overline{\T_{h_0}}$ as required in (H3).
Next choose $\eta$ with $\varepsilon<\eta<h_1$, and then choose
\(
0<\gamma<\min\{\varepsilon,h_1-\eta\}.
\)
Thus $0<\gamma<\varepsilon<\eta<h_1-\gamma$. Applying the construction
of \cite[Theorem~9.2]{LXZ} with these fixed parameters gives all the
assertions above.
\end{proof}

\begin{remark}
There is no essential distinction in this construction between
$\bar\omega=1$ and $\bar\omega=k$, or between the first turning point
and any later one. The chosen turning point is used only to provide the
strict gap condition that isolates the corresponding center bundle.
See \cite{LWYZ} for a
related discussion.
\end{remark}

Let $E_{\mathbf k}$ be the maximal energy satisfying $N_{v,\alpha}(E_{\mathbf k})\equiv\mathbf k\alpha\pmod{\Z}$ and
$E_{\mathbf k}^d$ denote the maximal energy with the same IDS label
for $v_d$. \cite[Lemma 9.1]{LXZ} gives
\begin{equation}\label{eq:ends}
|E_{\mathbf k}^d-E_{\mathbf k}|\leq\|v_d-v\|_0,
\end{equation}
so $E_{\mathbf k}^d\in\I$ for all sufficiently large $d$. 
In particular, convergence and the Hermitian symplecticity give
a constant $K_M\geq1$, independent of $E\in\I$ and $d$, such that
\begin{equation}\label{eq:uniform-M}
\|\mathcal M^{E,d}\|_0+\|(\mathcal M^{E,d})^{-1}\|_0\leq K_M.
\end{equation}
The dependence on $E$ is needed at the moving endpoints. Indeed,
\begin{equation}\label{eq:moving-center}
\|\mathcal M^{E_{\mathbf k}^d,d}
       -\mathcal M^{E_{\mathbf k}}\|_0
\leq\sup_{E\in\I}
   \|\mathcal M^{E,d}-\mathcal M^E\|_0 +\|\mathcal M^{E_{\mathbf k}^d}
                   -\mathcal M^{E_{\mathbf k}}\|_0
\longrightarrow0.
\end{equation}
Thus all the two-dimensional
iteration bounds used below are uniform at $E=E_{\mathbf k}^d$.

\section{Quantitative Aubry duality}

We now prove the analogue of \cite[Proposition~7.1]{LXZ}
with the original Schr\"odinger cocycle as input. The Fourier
transform then gives two approximate solutions of the long-range
dual equation. The purpose of the center construction is to carry
out the final two-column contradiction in dimension two.

\begin{proposition}\label{prop:subcritical-duality}
Let $\alpha,v,\I$ be as in Theorem~\ref{theorem-main-general}.
For every $0<\epsilon<h_1$ and $c>0$, there exist
$\delta_*>0$ and $q_*>0$ such that, for $q>q_*$, $E\in\I$ and
$d\geq d_0$, there are no $\sigma\in\{1,-1\}$ and
$U_d\in C^\omega_\epsilon(\T,\rmm{SL}(2,\R))$ satisfying
\begin{equation}\label{eq:forbidden}
\|S_E^{v_d}(\cdot)U_d(\cdot)
   -\sigma U_d(\cdot+\alpha)\|_\epsilon
\leq e^{(-c+\delta_*)q},\qquad\|U_d\|_\epsilon\leq e^{\delta_*q}.
\end{equation}
All constants are independent of $E$ and $d$.
\end{proposition}
\begin{proof}
We proceed by contradiction, following the two-column argument in \cite[Proposition~7.1]{LXZ}. Suppose \eqref{eq:forbidden}
holds. Throughout the proof, constants are independent of
$E,d,q$ and $0<\delta_*\leq1$; we choose $\delta_*$ at the end.

Write
\[
U_d(x)=\bigl(u_+^d(x),u_-^d(x)\bigr),\qquad
u_\pm^d(x)=
\begin{pmatrix}u_\pm^d(0,x)\\u_\pm^d(-1,x)\end{pmatrix}.
\]
Let $f_\pm^d$ be the columns of
$S_E^{v_d}U_d-\sigma U_d(\cdot+\alpha)$.
The two rows of the equation are
\begin{align}
(E-v_d(x))u_\pm^d(0,x)-u_\pm^d(-1,x)
 -\sigma u_\pm^d(0,x+\alpha)&=f_\pm^d(0,x),\notag\\
u_\pm^d(0,x)-\sigma u_\pm^d(-1,x+\alpha)
 &=f_\pm^d(-1,x). \label{eq:original-rows}
\end{align}
In particular,
\begin{equation}\label{eq:second-original}
u_\pm^d(-1,x)=
\sigma u_\pm^d(0,x-\alpha)-\sigma f_\pm^d(-1,x-\alpha).
\end{equation}
Since $\det U_d=1$ and $\|u_\pm^d\|_0\leq e^{\delta_*q}$,
integration gives
\[
1\leq e^{\delta_*q}
\bigl(\|u_\pm^d(0,\cdot)\|_{L^2}
     +\|u_\pm^d(-1,\cdot)\|_{L^2}\bigr).
\]
Using \eqref{eq:second-original}, for sufficiently small $\delta_*$,
\begin{equation}\label{eq:original-mass}
\|u_\pm^d(0,\cdot)\|_{L^2}
\geq C^{-1}e^{-\delta_*q}.
\end{equation}

Define
\[
\widehat u_\pm^d(k)
:=\widehat{(u_\pm^d(0,\cdot))}_{-k},
\qquad
u_\pm^d(0,x)=\sum_{k\in\Z}\widehat u_\pm^d(k)e^{-2\pi i kx}.
\]
Eliminating the second component from \eqref{eq:original-rows}
gives
\[
(v_d(x)-E)u_\pm^d(0,x)
+\sigma u_\pm^d(0,x-\alpha)+\sigma u_\pm^d(0,x+\alpha)
=-f_\pm^d(0,x)+\sigma f_\pm^d(-1,x-\alpha).
\]
Let $\theta=0$ if $\sigma=1$, and $\theta=1/2$ if $\sigma=-1$.
Taking Fourier coefficients yields the dual equation
\begin{equation}\label{eq:dual-scalar}
\sum_{j=-d}^{d}\widehat v_j\widehat u_\pm^d(k+j)
+\bigl(2\cos 2\pi(\theta+k\alpha)-E\bigr)
                 \widehat u_\pm^d(k)
=r_\pm^d(k).
\end{equation}
Thus $(L_{v_d,\alpha,\theta}-E)\widehat u_\pm^d=r_\pm^d$.
Analyticity and Parseval's identity give
\begin{align}
|\widehat u_\pm^d(k)|
&\leq C e^{\delta_*q}e^{-2\pi \epsilon |k|},\qquad
|r_\pm^d(k)|
\leq C e^{(-c+\delta_*)q}e^{-2\pi \epsilon |k|}, \label{eq:Fourier-bounds}\\
\sum_{k\in\Z}|\widehat u_\pm^d(k)|^2
&=\|u_\pm^d(0,\cdot)\|_{L^2}^2
 \geq C^{-1}e^{-2\delta_*q}. \label{eq:Fourier-mass}
\end{align}
These are the two nontrivial, exponentially decaying approximate
solutions to which we apply the dual center argument.

Set
\[
\begin{aligned}
\mathcal U_k^\pm=
(
\widehat u_\pm^d(k+d-1),\dots, \widehat u_\pm^d(k-d)
)^\top,\ \mathcal V_k^\pm
&=\mathcal J_2^{-1}
\mathcal O^{E,d}(\theta+k\alpha)^*S_d\mathcal U_k^\pm.
\end{aligned}
\]
Here $\mathcal O^{E,d}=(\mathcal O_+^{E,d},\mathcal O_-^{E,d})$
denotes its two columns.
Note that \[\mathcal V_k^\pm=\begin{pmatrix}
\psi_d\bigl(\mathcal O_-^{E,d}(\theta+k\alpha),\mathcal U_k^\pm\bigr)\\
-\psi_d\bigl(\mathcal O_+^{E,d}(\theta+k\alpha),\mathcal U_k^\pm\bigr)
\end{pmatrix}.\]
Thus $\mathcal V_k^\pm\in\C^2$ is the coordinate representation
of the center component of $\mathcal U_k^\pm\in\C^{2d}$ with
respect to the Hermitian symplectic frame
$\mathcal O^{E,d}(\theta+k\alpha)$.

The following lemma is the major adjustment from the  standard proof.
\begin{lemma}\label{lem:approximate-center}
There are $a_0,C>0$, independent of $E,d,q$, such that 
\begin{align}
\mathcal V_{k+1}^\pm
&=\mathcal M^{E,d}(\theta+k\alpha)\mathcal V_k^\pm
 +\mathcal R_k^\pm, \label{eq:dual-center-rec}\\
\|\mathcal V_k^\pm\|
&\leq C e^{\delta_*q}e^{-a_0|k|},\qquad
\|\mathcal R_k^\pm\|
\leq C e^{(-c+\delta_*)q}e^{-a_0|k|},
\label{eq:center-bounds}\\
\widehat u_\pm^d(k)
&=\mathcal O^{E,d}(\theta+k\alpha,0)\mathcal V_k^\pm
 +\zeta_\pm^d(k),\qquad
|\zeta_\pm^d(k)|
\leq C e^{(-c+\delta_*)q}e^{-a_0|k|}.
\label{eq:recover-columns}
\end{align}
\end{lemma}

\begin{proof}
We use properties (H1), (H3), and (H4) of
Theorem~\ref{theorem-main-general}.
All estimates here hold uniformly for $E\in\I$ and
sufficiently large $d$.
Fix a sign and suppress it from the notation. Choose
$0<a_0<2\pi\min\{\epsilon,\varepsilon\}$.

By (H1) and the definition of the weighted norm,
\begin{equation}\label{eq:frame-exponential}
\|\mathcal O^{E,d}(\cdot,j)\|_0
\leq K e^{2\pi\eta|j|},
\qquad -d\leq j\leq d-1.
\end{equation}
Only this fixed exponential bound is needed.

From $(L_{v_d,\alpha,\theta}-E)\widehat u=r$ we obtain
\[
\mathcal U_{k+1}
=L^{E,d}(\theta+k\alpha)\mathcal U_k
+\frac{r(k)}{\widehat v_d}e_1.
\]
Define
\(
\mathcal P^{E,d}(x)
=\mathcal J_2^{-1}\mathcal O^{E,d}(x)^*S_d
\).
Hermitian symplectic invariance and \eqref{eq:center-invariance} imply
\[
\mathcal P^{E,d}(x+\alpha)L^{E,d}(x)
=\mathcal M^{E,d}(x)\mathcal P^{E,d}(x).
\]
Since $S_de_1=\widehat v_d e_{d+1}$, with $e_{d+1}$ the row $-1$,
the projected equation is exactly
\begin{equation}\label{eq:forcing}
\mathcal V_{k+1}
=\mathcal M^{E,d}(\theta+k\alpha)\mathcal V_k
+\mathcal J_2^{-1}
 \mathcal O^{E,d}(\theta+(k+1)\alpha,-1)^*r(k).
\end{equation}
Thus $1/\widehat v_d$ disappears from the error term.
By \eqref{eq:frame-exponential} and \eqref{eq:Fourier-bounds},
the last term in \eqref{eq:forcing} satisfies
\(
\|\mathcal R_k\|\leq
C e^{(-c+\delta_*)q}e^{-a_0|k|}.
\)

For the coordinate bound, use the entry formula above:
\[
(S_d)_{j,\ell}
=(\mathbf1_{j<0}-\mathbf1_{\ell<0})
\widehat v_{\ell-j}\mathbf1_{|\ell-j|\leq d}.
\]
Since $v\in C_{h_1}^\omega(\T)$,
$|\widehat v_k|\leq\|v\|_{h_1}e^{-2\pi h_1|k|}$.
Only indices across the origin contribute; then
$|j-\ell|=|j|+|\ell|$. Consequently,
\[
\|\mathcal P^{E,d}(x)e_\ell\|
\leq C\sum_{j\text{ across the origin from }\ell}
e^{2\pi\eta|j|}e^{-2\pi h_1(|j|+|\ell|)}\leq C'e^{-2\pi h_1|\ell|}.
\]
Summing against $\widehat u(k+\ell)$ and using
\eqref{eq:Fourier-bounds} proves
$\|\mathcal V_k\|\leq Ce^{\delta_*q}e^{-a_0|k|}$.

Let $\mathcal K^{E,d}$ be the kernel constructed in
Proposition~\ref{prop:backward-cut-kernel} by factoring
the complement of the center projection through
$L_{v_d,\alpha,\theta}-E$. It gives
\begin{equation}\label{eq:kernel-decay}
|\mathcal K^{E,d}(\theta;k,m)|
\leq C e^{-2\pi\varepsilon|k-m|}
\end{equation}
and the exact reconstruction formula
\begin{equation}\label{eq:exact-reconstruction}
\widehat u(k)
=\mathcal O^{E,d}(\theta+k\alpha,0)\mathcal V_k
+\sum_{m\in\Z}\mathcal K^{E,d}(\theta;k,m)r(m).
\end{equation}
Convolution with
\eqref{eq:Fourier-bounds} gives
\[
\left|\sum_m\mathcal K^{E,d}(\theta;k,m)r(m)\right|
\leq C e^{(-c+\delta_*)q}e^{-a_0|k|}
\]
by the choice of $a_0$.
Together with \eqref{eq:forcing}, this proves
Lemma~\ref{lem:approximate-center}.
\end{proof}

Now, by \eqref{eq:Fourier-mass}, \eqref{eq:recover-columns} and the
uniform bound on $\mathcal O^{E,d}(\cdot,0)$ imply
\[
\sum_k\|\mathcal V_k^\pm\|^2\geq C^{-1}e^{-2\delta_*q}.
\]
The tails in \eqref{eq:center-bounds} can be discarded outside
$|k|\leq C\delta_*q+O(\log q)$. Hence for each sign there is a
$k_\pm$ in this interval with
$\|\mathcal V_{k_\pm}^\pm\|\geq e^{-C\delta_*q-o(q)}$.
Using \eqref{eq:uniform-M} in \eqref{eq:dual-center-rec}, in either
time direction, gives
\begin{equation}\label{eq:center-mass-zero}
\|\mathcal V_0^\pm\|\geq e^{-C\delta_*q-o(q)}.
\end{equation}

Set
\[
\mathcal V_k=(\mathcal V_k^+,\mathcal V_k^-),
\qquad D_k=\det\mathcal V_k.
\]
Expanding the determinant in \eqref{eq:dual-center-rec} gives
\begin{equation}\label{eq:det-rec}
D_{k+1}=\det\bigl(\mathcal M^{E,d}(\theta+k\alpha)\bigr)D_k+\eta_k,
\qquad
|\eta_k|\leq C e^{(-c+C\delta_*)q}e^{-2a_0|k|}.
\end{equation}
We use $|\det\mathcal M^{E,d}|=1$, which follows from Hermitian
symplecticity. Since $D_N\to0$ as $N\to+\infty$ by
\eqref{eq:center-bounds}, iteration of \eqref{eq:det-rec} yields
\begin{equation}\label{eq:det-small}
|D_0|\leq\sum_{k=0}^{\infty}|\eta_k|
\leq C e^{(-c+C\delta_*)q}.
\end{equation}
Thus the two center vectors at $k=0$ are almost collinear.

Choose $\gamma_0\in\C$ so that
$\mathcal V_0^+-\gamma_0\mathcal V_0^-$ is perpendicular to
$\mathcal V_0^-$. Equations \eqref{eq:center-mass-zero} and
\eqref{eq:det-small} imply
\begin{align}
|\gamma_0|&\leq e^{C\delta_*q+o(q)},\qquad
\|\mathcal V_0^+-\gamma_0\mathcal V_0^-\|
=\frac{|\det\mathcal V_0|}{\|\mathcal V_0^-\|}
 \leq e^{(-c+C\delta_*)q+o(q)}. \label{eq:combination-zero}
\end{align}
Subtract $\gamma_0$ times the minus equation from the plus equation
in \eqref{eq:dual-center-rec}. The same $\gamma_0$, independent of
$k$, then satisfies
\[
\begin{split}
\mathcal V_{k+1}^+-\gamma_0\mathcal V_{k+1}^-
={}&\mathcal M^{E,d}(\theta+k\alpha)
       (\mathcal V_k^+-\gamma_0\mathcal V_k^-)+\mathcal R_k^+-\gamma_0\mathcal R_k^-.
\end{split}
\]
Consequently, for a fixed $C_0>0$ depending on $K$,
\begin{equation}\label{eq:combination-propagation}
\|\mathcal V_k^+-\gamma_0\mathcal V_k^-\|
\leq e^{(-c+C\delta_*)q+C_0|k|+o(q)}.
\end{equation}

For $|k|\leq \delta_0q$, recover the scalar coefficients using
\eqref{eq:recover-columns} and \eqref{eq:combination-propagation}.
For $|k|>\delta_0q$, use the Fourier tails in \eqref{eq:Fourier-bounds}.
It follows that
\[
\begin{split}
\sum_{k\in\Z}
|\widehat u_+^d(k)-\gamma_0\widehat u_-^d(k)|
\leq{}&
 e^{(-c+C\delta_*+C_0\delta_0)q+o(q)}+e^{(-2\pi\epsilon\delta_0+C\delta_*)q+o(q)}.
\end{split}
\]
First choose $\delta_0>0$ with $C_0\delta_0<c/4$ and set
$c_1=\min\{c/16,2\pi\epsilon\delta_0/8\}$. Then choose $\delta_*>0$ so small that
all accumulated losses $C\delta_*$ are smaller than
$\min\{c/4,2\pi\epsilon\delta_0/4,c_1/2\}$. Make this choice also small enough to
absorb the errors in \eqref{eq:original-mass} and
\eqref{eq:center-mass-zero}. For large $q$, the right-hand side
is at most $e^{-2c_1q}$. Fourier inversion gives
$\|u_+^d(0,\cdot)-\gamma_0u_-^d(0,\cdot)\|_0\leq e^{-2c_1q}.$ 
Equation \eqref{eq:second-original} gives a bound of the same
exponential order for the component $-1$.
Therefore
$\|u_+^d-\gamma_0 u_-^d\|_0 \leq e^{-c_1q}.$
For every $x$, the determinant normalization now yields
$$
1=|\det U_d(x)|
=|\det(u_+^d(x)-\gamma_0u_-^d(x),u_-^d(x))|
\leq e^{-c_1q}e^{\delta_*q}<1,
$$
where $\delta_*$ was chosen also smaller than $c_1/2$.
This is the desired contradiction.
\end{proof}

\section{Quantitative almost reducibility}

We now return to $E_{\mathbf k}^d\to E_{\mathbf k}$. Apply
Theorem~8.8 or Theorem
9.8 in \cite{LXZ} to the original cocycle
$(\alpha,S_{E_{\mathbf k}}^v)$. Using their notation, there are
a fixed $\gamma>0$, conjugacies $\bar B_{q_n} \in C^\omega_\gamma(\mathbb{T},\mathrm{PSL}(2,\mathbb{R}))$ with $\deg \bar B_{q_n}=2\mathbf{k}$\footnote{In the sense of \cite[Definition A.1]{LW}.} and real constants
$c_{q_n}$ such that
\begin{equation}\label{eq:AR}
(-1)^{\mathbf l}\bar B_{q_n}(x+\alpha)^{-1}
 S_{E_{\mathbf k}}^v(x)\bar B_{q_n}(x)
=I_2+c_{q_n}\mathfrak L+F_{q_n}(x),
\qquad
\mathfrak L=\begin{pmatrix}0&1\\0&0\end{pmatrix},
\end{equation}
where
\begin{equation}\label{eq:AR-bounds}
\|\bar B_{q_n}\|_\gamma\leq e^{o(q_n)},
\qquad
\|F_{q_n}\|_\gamma\leq e^{-\kappa q_n}.
\end{equation}
Write
\begin{equation}\label{eq:AR-d}
(-1)^{\mathbf l}\bar B_{q_n}(x+\alpha)^{-1}
 S_{E_{\mathbf k}^d}^{v_d}(x)\bar B_{q_n}(x)
=I_2+c_{q_n}\mathfrak L+F_{q_n}(x)+Q^d_{q_n}(x).
\end{equation}
Note by \eqref{eq:ends}, 
\begin{equation}\label{eq:Qd}
\begin{split}
\|Q^d_{q_n}\|_\gamma
&\leq\|\bar B_{q_n}\|_\gamma^2
 (|E_{\mathbf k}^d-E_{\mathbf k}|+\|v_d-v\|_\gamma)\leq2\|\bar B_{q_n}\|_\gamma^2\|v_d-v\|_\gamma.
\end{split}
\end{equation}
For each $q_n$, choose $d(q_n)$ so large that $
\|Q^d_{q_n}\|_\gamma\leq e^{-\kappa' q_n}$ for every $d\geq d(q_n)$.

\begin{lemma}\label{lem:parabolic}
For sufficiently large $q_n$ and for all sufficiently large $d$, we have
\(
|c_{q_n}|>e^{-o(q_n)}.
\)
\end{lemma}

\begin{proof}
Set
\(
U_d(x)=\bar B_{q_n}(x)\), \(\sigma=(-1)^{\mathbf l}.
\)
Multiplying \eqref{eq:AR-d} by $\bar B_{q_n}(x+\alpha)$ gives
\[
\begin{split}
S_{E_{\mathbf k}^d}^{v_d}(x)U_d(x)-\sigma U_d(x+\alpha)
=\sigma U_d(x+\alpha)
 \bigl(c_{q_n}\mathfrak L+F_{q_n}(x)+Q^d_{q_n}(x)\bigr).
\end{split}
\]
Suppose $|c_{q_n}|\leq e^{-\eta q_n}$ for some fixed
$\eta>0$ along a subsequence. Then for $d\geq d(q_n)$,
\[
\|U_d\|_\gamma\leq e^{o(q_n)},\qquad
\|S_{E_{\mathbf k}^d}^{v_d}U_d-\sigma U_d(\cdot+\alpha)\|_\gamma
\leq e^{-\min\{\eta,\kappa,\kappa'\}q_n+o(q_n)}.
\]
This contradicts Proposition~\ref{prop:subcritical-duality}.
\end{proof}

After a constant diagonal conjugacy of subexponential size \cite[Remark 9.9]{LXZ}, we may
assume $0<|c_{q_n}|\leq1$, retaining \eqref{eq:AR-bounds}.
As in \cite{LXZ}, the lower bound is always understood
together with the bound on the chosen conjugacy.

\subsection{Proof of Theorem~\ref{thm:subcritical}}
It then follows by the generalized Moser-P\"oschel argument or the cone argument \cite{AYZ} that the gap length are subexponential in $q_n$ and stablized. We complete the proof by the continuity of the gap ends on $v$, for example, \eqref{eq:ends}.

\appendix

\section{Hyperbolic projection}
\label{app:backward-cut-kernel}

Fix $E,\theta,d$ and write
$H_d=L_{v_d,\alpha,\theta}-E$, $x_k=\theta+k\alpha$. 
\begin{proposition}\label{prop:backward-cut-kernel}
Let $\varepsilon$ be the fixed shift chosen in Theorem \ref{theorem-main-general}, we have
\begin{equation}\label{eq:backward-factorization}
\delta_k^*\bigl(I-\Lambda_kG_d(E,x_k)S_d\Lambda_k^*\bigr)
=\mathcal K^{E,d}(\theta;k,\cdot)H_d,
\end{equation}
where
\begin{equation}\label{eq:cut-kernel-definition}
\mathcal K^{E,d}(\theta;k,m)
:=\begin{cases}
e^{-2\pi\varepsilon(k-m)}g_{v_d(\cdot+i\varepsilon)}^m(E,\theta,k),&m<k,\\
e^{2\pi\varepsilon(k-m)}g_{v_d(\cdot-i\varepsilon)}^m(E,\theta,k),&m\geq k.
\end{cases}
\end{equation}
Moreover,
\begin{equation}\label{eq:cut-kernel-decay}
|\mathcal K^{E,d}(\theta;k,m)|\leq K e^{-2\pi\varepsilon|k-m|}
\end{equation}
\end{proposition}

\begin{proof}
We use the resolvent entries and $u_d^\ell$ defined in Theorem \ref{theorem-main-general}(H3)--(H4).
By translation covariance,
\begin{equation}\label{eq:G-row}
(G_d)_{0,j}(E,x_k)=-u_d^{k+j}(E,\theta,k).
\end{equation}
The resolvent identities give, entrywise,
\begin{equation}\label{eq:two-formal-inverses}
\begin{aligned}
\sum_m e^{\mp 2\pi\varepsilon(k-m)}
g_{v_d(\cdot\pm i\varepsilon)}^m(E,\theta,k)H_d(m,\ell)
=\delta_{k\ell},\qquad\sum_m u_d^m(E,\theta,k)H_d(m,\ell)=0.
\end{aligned}
\end{equation}

Put $\chi_k(m)=\mathbf1_{m<k}$. Note that
\[
(S_d)_{m-k,\ell-k}
=(\chi_k(m)-\chi_k(\ell))H_d(m,\ell)
=(\mathbf1_{m<k}-\mathbf1_{\ell<k})
\widehat v_{\ell-m}\mathbf1_{|\ell-m|\leq d}
\]
with both sides understood to be zero outside the window. Therefore,
using \eqref{eq:G-row}, \eqref{eq:two-formal-inverses},
\[
\begin{aligned}
&\bigl[
\delta_k^*
\bigl(I-\Lambda_kG_d(E,x_k)S_d\Lambda_k^*\bigr)
\bigr](\ell)\\
&\quad=
\bigl[
e_d^*\bigl(I_{2d}-G_d(E,x_k)S_d\bigr)\Lambda_k^*
\bigr](\ell)\\
&\quad=\delta_{k\ell}
+\sum_m u_d^m(E,\theta,k)
(\chi_k(m)-\chi_k(\ell))H_d(m,\ell)\\
&\quad=\sum_m\Bigl[
e^{2\pi\varepsilon(k-m)}g_{v_d(\cdot-i\varepsilon)}^m(E,\theta,k)
+\chi_k(m)u_d^m(E,\theta,k)\Bigr]H_d(m,\ell).
\end{aligned}
\]
Reading off the factor in brackets and substituting the definition
of $u_d^\ell$ gives \eqref{eq:cut-kernel-definition}, and hence
\eqref{eq:backward-factorization}.
The bound in (H3) gives \eqref{eq:cut-kernel-decay}.
\end{proof}
\begin{remark}
   Although the left side of \eqref{eq:backward-factorization} is
independent of $\varepsilon$, its factorization through $H_d$ need not
be unique. For two admissible shifts $\varepsilon$ and $\varepsilon'$,
the corresponding kernels satisfy
$(\mathcal K_\varepsilon^{E,d}-\mathcal K_{\varepsilon'}^{E,d})H_d=0$.
\end{remark}


\begin{thebibliography}{99}

\bibitem{AA}
F. Argentieri and A. Avila,
in preparation.

\bibitem{Avi2023KAM}
A. Avila,
\emph{KAM, Lyapunov exponents, and the spectral dichotomy for typical
one-frequency Schr\"odinger operators},
preprint (2023), arXiv:2307.11071.

\bibitem{ABD}
A. Avila, J. Bochi, and D. Damanik,
\emph{Opening gaps in the spectrum of strictly ergodic Schr\"odinger
operators},
J. Eur. Math. Soc. \textbf{14} (2012), 61--106.

\bibitem{AJ05}
A. Avila and S. Jitomirskaya,
\emph{The Ten Martini Problem},
Ann. of Math. \textbf{170} (2009), 303--342.

\bibitem{AJ08}
A. Avila and S. Jitomirskaya,
\emph{Almost localization and almost reducibility},
J. Eur. Math. Soc. \textbf{12} (2010), 93--131.

\bibitem{AYZ}
A. Avila, J. You, and Q. Zhou,
\emph{Dry Ten Martini Problem in the non-critical case},
preprint (2024), arXiv:2306.16254v2.

\bibitem{BBL}
R. Band, S. Beckus, and R. Loewy,
\emph{The Dry Ten Martini Problem for Sturmian Hamiltonians},
preprint (2024), arXiv:2402.16703.

\bibitem{CEY}
M. D. Choi, G. A. Elliott, and N. Yui,
\emph{Gauss polynomials and the rotation algebra},
Invent. Math. \textbf{99} (1990), 225--246.

\bibitem{DGY}
D. Damanik, A. Gorodetski, and W. Yessen,
\emph{The Fibonacci Hamiltonian},
Invent. Math. \textbf{206} (2016), 629--692.

\bibitem{GJ}
L. Ge and S. Jitomirskaya,
\emph{Intrinsic symplectic structure and sharp arithmetic universality},
preprint (2026), arXiv:2407.08866v2.

\bibitem{GJY}
L. Ge, S. Jitomirskaya, and J. You,
\emph{Kotani theory, Puig's argument, and stability of The Ten Martini
Problem},
preprint (2023), arXiv:2308.09321.

\bibitem{GJY2}
L. Ge, S. Jitomirskaya, and J. You,
\emph{The Robust Ten Martini Problem},
preprint (2026), arXiv:2609.15812v1.

\bibitem{GeWXu}
L. Ge, Y. Wang, and J. Xu,
\emph{The Dry Ten Martini Problem for $C^2$ cosine-type quasiperiodic
Schr\"odinger operators},
preprint (2025), arXiv:2503.06918.

\bibitem{gaplabel}
R. Johnson and J. Moser,
\emph{The rotation number for almost periodic potentials},
Comm. Math. Phys. \textbf{84} (1982), 403--438.

\bibitem{LWYZ}
X. Li, Z. Wang, J. You, and Q. Zhou,
\emph{Transfer operators, canonical center dynamics, and spectral
applications for long-range operators},
preprint (2026), arXiv:2606.29154v1.

\bibitem{LW}
X. Li and L. Wu,
\emph{The fibered rotation number for ergodic symplectic cocycles and its
applications: I. Gap labelling theorem},
Math. Z. \textbf{311} (2025), Art.~53.

\bibitem{LXZ}
X. Li, D. Xu, and Q. Zhou,
\emph{Monotonicity, global symplectification and the stability of Dry Ten
Martini Problem},
preprint (2026), arXiv:2601.02222v3.

\bibitem{LY2015}
W. Liu and X. Yuan,
\emph{Spectral gaps of almost Mathieu operators in the exponential regime},
J. Fractal Geom. \textbf{2} (2015), no.~1, 1--51.

\bibitem{P}
J. Puig,
\emph{Cantor spectrum for the almost Mathieu operator},
Comm. Math. Phys. \textbf{244} (2004), 297--309.

\bibitem{P06}
J. Puig,
\emph{A nonperturbative Eliasson's reducibility theorem},
Nonlinearity \textbf{19} (2006), 355--376.

\bibitem{You}
J. You,
\emph{Some problems in quasiperiodic Schr\"odinger operators},
J. Math. Phys. \textbf{67} (2026), 062704.

\end{thebibliography}
\end{document}